\documentclass[11pt,twoside,reqno]{amsart}
\usepackage{times}
\usepackage{array} 
\usepackage{mathtools}
\usepackage{mathrsfs}    
\usepackage{amssymb,amsmath,amsfonts}
\usepackage{algorithm}
\usepackage{multirow}
\usepackage{xcolor}
\usepackage{algorithmicx}
\makeatletter
\newenvironment{breakablealgorithm}
  {% \begin{breakablealgorithm}
   \begin{center}
     \refstepcounter{algorithm}% New algorithm
     \hrule height.8pt depth0pt \kern2pt% \@fs@pre for \@fs@ruled 画线
     \renewcommand{\caption}[2][\relax]{% Make a new \caption
       {\raggedright\textbf{\ Algorithm:~~\thealgorithm} ##2\par}%
       \ifx\relax##1\relax % #1 is \relax
         \addcontentsline{loa}{algorithm}{\protect\numberline{\thealgorithm}##2}%
       \else % #1 is not \relax
         \addcontentsline{loa}{algorithm}{\protect\numberline{\thealgorithm}##1}%
       \fi
       \kern2pt\hrule\kern2pt
     }
  }{% \end{breakablealgorithm}
     \kern2pt\hrule\relax% \@fs@post for \@fs@ruled 画线
   \end{center}
  }
\makeatother

\usepackage[colorlinks = true,
            linkcolor = blue,
            urlcolor  = violet, 
            citecolor = blue,
            anchorcolor = blue]{hyperref}
            
\newtheorem{thm}{Theorem}[section]

\theoremstyle{definition}
\newtheorem{definition}[thm]{Definition}
\newtheorem{lemma}[thm]{Lemma}%%%%%%%%%%%%%%%%%%%%%%%%%%%%

\theoremstyle{remark}
\newtheorem{remark}[thm]{Remark}
\newtheorem{case}{Case}
\newtheorem{assumption}{Assumption}
\numberwithin{equation}{section} 
\newtheorem{theorem}{Theorem}[section]
\makeatletter
\newcommand{\thickhline}{%
    \noalign {\ifnum 0=`}\fi \hrule height 1pt
    \futurelet \reserved@a \@xhline
}
\begin{document}

\title[solving quasimonotone variational inequality problem]{A Two-Step Inertial Method with a New Step-Size Rule for  Variational Inequalities in Hilbert Spaces}

\author[J. Peng, J. Luo, and A. Adamu]{Jian-Wen Peng$^{1}$, Jun-Jie Luo$^{1}$, and Abubakar Adamu$^{1,2,*}$}
\maketitle
\vspace*{-0.6cm}

\begin{center}
{\footnotesize {\it
$^1$ School of Mathematical Sciences, Chongqing Normal University, Chongqing 401331, People’s Republic of China \\ 

%$^2$ School of Mathematical Sciences, Chongqing Normal University, Chongqing 401331, People’s Republic of China \\ 

$^2$Near East University, TRNC Mersin 10, Nicosia 99138, Turkey.  \\ 
}}\end{center}

\begin{abstract}
In this paper, a two-step inertial Tseng extragradient method involving self-adaptive and Armijo-like step sizes is introduced for solving variational inequalities with a quasimonotone cost function in the setting of a real Hilbert space. Weak convergence of the sequence generated by the proposed algorithm is proved without assuming the Lipschitz condition. An interesting feature of the proposed algorithm is its ability to select the better step size between the self-adaptive and Armijo-like options at each iteration step. Moreover, removing the requirement for the Lipschitz condition on the cost function broadens the applicability of the proposed method. Finally, the algorithm accelerates and complements several existing iterative algorithms for solving variational inequalities in Hilbert spaces.
\end{abstract}

\maketitle


\renewcommand{\thefootnote}{}
\footnotetext{ $^*$Corresponding author.
\par
E-mail addresses: jwpeng6@yahoo.com.cn (J. Peng), 3589995284@qq.com.cn  (J. Luo), abubakar.adamu@neu.edu.tr(A. Adamu)
%\url{www.ams.org/msc} and include at least two subject classification codes. 
}

%%%%%%%%%%%%%%%%%%%%%%%%%%%%%%%%%%%%%%%%%%%%%%%%%%%%%%%%%%%%%%%%%%%%%%
\section{Introduction}
%%%%%%%%%%%%%%%%%%%%%%%%%%%%%%%%%%%%%%%%%%%%%%%%%%%%%%%%%%%%%%%%%%%%%%
Let $H$ be a real Hilbert space with the inner product$\langle\cdot,\cdot\rangle$ and the induce norm$\|\cdot\|$. Then we suppose that  $C$ is nonempty, closed and convex subset of $H$ ($C\subset H$). In this work, we focus on the following problem: 
\begin{equation}\label{In1}
\mbox{Find } ~~ x^*\in C~~ \mbox{  such that } ~~ \langle Ax^*,\:x-x^*\rangle\geq0,\:\forall x\in C.
\end{equation}
Problem (\ref{In1}) is popularly known as the  variational inequality problem (VIP) ($VI(C,A),$ for short). Let $\Omega$ denote  the set of solutions of $VI(A, C)$ (\ref{In1}).  That is, $\Omega:=\{q \in C: \langle A(q),x-q\geq0 \rangle\}$.  The VIP has been used as a modeling tool for the study of various real-life problems such as the obstacle problem, contact problem, traffic network problem, and optimal control problem, and have also been applied in diverse areas of study such as optimization theory, nonlinear analysis, and computational mechanics. Those applications of VI(A, C) (\ref{In1}) mentioned about are discussed in (\cite{aplct 2},\cite{aplct 3},\cite{aplct 4},\cite{aplct 5},\cite{aplct 6}).

 Numerous iterative algorithms for solving  $VI(A, C)$ (\ref{In1}) have been extensively studied and developed by many authors (see, for example \cite{ex1}, \cite{ex2},\cite{ex4},\cite{ex8} and the references therein). The fundamental idea involves extending the projected gradient method, originally designed to solve a constrained minimization problem involving $f$ over some nonempty closed and convex set say,   $C$. The iterative procedure is  given by:
\begin{equation}\label{xnp1=pc}
\begin{cases}
  x_1\in \mathbb{R}^n,\\
  x_{k+1} = P_C \left(x_k - \alpha_k \nabla f(x_k) \right), 
\end{cases} 
\end{equation}
where $P_C$ is the metric projection onto $C$,  ${\alpha_k}$ is a  positive real sequence that  satisfies some certain conditions and $\nabla f$ is the gradient of the smooth function $f$ \cite{projected gradient}. One of the early generalizations of the projected gradient method to the setting of VIP   was the method proposed by Goldstein \cite{1st method}:
\begin{equation}\label{1st method}
\begin{cases}
  x_1\in \mathbb{R}^n\\
  x_{k+1}=P_{C}(x_{k}-\lambda Ax_{k}), 
\end{cases} 
\end{equation}
 where $\lambda$ is a positive real number, $P_C$ is the projection operator and $A$ is a given mapping. But the convergence of this method is only guaranteed under a highly restrictive condition  that the operators are strongly monotone or inverse strongly monotone. (see for\cite{quedian of goldstein} )

 Korpelevich \cite{Korpelevich} improved Goldstein's method by extending the operator  $A$ to monotone and $L$-Lipschitz continuous mapping. He then introduced the famous  extragradient method (EGM) given by:
 \begin{equation}\label{3method EGM}
 \begin{cases}
   x_1 \in \mathbb{R}^n,\\
   y_k = P_C(x_k-\lambda Ax_k), \\
   x_{k+1} = P_C(x_k-\lambda Ay_k),
   \end{cases}
 \end{equation}
 where  $\lambda \in (0,1/L)$. 
\begin{remark} \label{rmk1}
It is well-known that one of the drawbacks of the EGM arises from its requirement to perform two projections onto the closed convex set during each iteration. This requirement can be computationally expensive especially when the structure of $C$ is not simple. Another drawback of this method is the fact that parameter $\lambda$ depends on the explicit value of the Lipschitz constant $L$ which is can be challenging to obtain. 
 \end{remark}
 
 Many authors have tried to address Remark \ref{rmk1} in different directions (see \cite{cite1},\cite{cite2},\cite{cite3},\cite{cite4},\cite{cite5},\cite{cite6},\cite{cite7},\cite{cite8},\cite{cite9}). Censor \cite{censor} introduced a new method called  the subgradient extragradient method (SEGM) which is defined as follows:
 
 \[\left\{ {\begin{array}{*{20}{l}}
{{x_1} \in H,}\\
{{y_k} = {P_C}\left( {{x_k} - \lambda A{x_k}} \right),}\\
{{T_n} = \left\{ {w \in H|\left\langle {{x_k} - \lambda A{x_k} - {y_k},w - {y_k}} \right\rangle  \le 0} \right\},}\\
{{x_{k + 1}} = {P_{{T_{\rm{k}}}}}({x_k} - \lambda A{y_k}),{\rm{ }}k > 1,}
\end{array}} \right.\]

  Another modification of the EGM was  the method introduced by   Tseng \cite{tseng}. His idea was to replace the second projection onto $C$ by a function evaluation. His algorithm is the following:
 \begin{equation}\label{tseng}
 \begin{cases}
 x_1\in H,\\
  y_{k}=P_{C}\left(x_{k}-\lambda Ax_{k}\right), \\
  x_{k+1}=y_{k}-\lambda\left(Ay_{k}-Ax_{k}\right),
  \end{cases}
 \end{equation}
 where $\lambda \in (0, 1/L)$. Tseng \cite{tseng} proved that the sequence generated by \cite{tseng} converges weakly to a point in $\Omega$.   The advantage of the above Tseng’s method is that it requires only  one computation of  projection onto the feasible set $C$ and two evaluations of $A$ per iteration.  
 
 It is worthy of mention that the sequences generated by all the algorithms above may have slow convergence properties. To accelerate the convergence, several authors have adopted the inertial acceleration technique which dates back to the early work of Polyak \cite{polyak} in the setting of convex minimization.   Alvarez and Attouch \cite{iner} adopted this  principle and extended it   to general maximal monotone operators through a proximal point framework. They proposed a new algorithm called inertial proximal point algorithm which is defined as:
 \begin{equation}\label{iner}
 \begin{cases}
 x_0,x_1\in H,\\
   y_k = x_k+\alpha_k(x_k-x_{k-1}), \\
   x_{k+1} = (I+\lambda_{k}A)^{-1}y_{k},
   \end{cases}
 \end{equation}
 where $(I+\lambda_k A)^{-1}$ is the resolvent of the maximal monotone operator $A$ and  $\lambda_k$ is a positive sequence that satisfies some appropriate conditions. Then, the authors proved that the sequence ${x_k}$ generated by(\ref{iner}) converges  weakly  to a zero of $A$ provided $\alpha_n \in [0,1)$ satisfies the following condition:
 \begin{equation}\label{condi of iner alpha}
   \sum_{k=1}^{\infty}\alpha_k\|x_k-x_{k-1}\|^2<+\infty.
 \end{equation}

In the setting of VIP, Thong and Hieu \cite{thonghieu} introduced an inertial step in the   Tseng's method for a better performance  in a real Hilbert space. Their algorithm is defined as:
\begin{equation}\label{Thong}
\begin{cases}
x_0,x_1\in H,\\
  w_k = x_k + \alpha_k(x_k + x_{k+1}),   \\
  y_k = P_C(w_k-\lambda_k w_k),   \\
  x_{k+1}= y_k - \lambda_k(Ay_k-Aw_k),
  \end{cases}
\end{equation}
where  A is monotone and Lipschitz continuous and $\lambda_k$ is a step-size obtained using   Armijo-like step size rule. \\

Some results from using one-step inertial have shown that algorithms with this acceleration technique may fail to outperform their counterpart that does not involve this step.  A counter example was given in \cite{quedian1} which which one-step inertial extrapolation fails to provide acceleration. Polyak mentioned  in \cite{2step youdian} that the use of inertial of more than two points $x_k,x_{k-1}$ could provide acceleration. Polyak \cite{2step youdian} also discussed that the multi-step inertial methods can boost the speed of optimization methods though neither the convergence nor the rate of such multi-step inertial methods was established in \cite{2step youdian}. Recently, several authors have explore the concept of two-step inertia to accelerate convergence (see for example, \cite{bang1}, \cite{okeke}, \cite{citetwo3}).

\begin{remark}
It is important to note that all the improvements of the EGM and their accelerated versions mentioned above have not fully addressed Remark \ref{rmk1}. They are yet to dispense with the dependency of the step-size on the explicit value of the Lipschitz constant.
\end{remark}  
%{\color{red} Introduce adaptive step-size the first paper by Liu and Yang and Introduce the  Armijo linesearch. This two are the methods used for dispensing with dependency of the step size on the Lipschitz constant}
In recent years, different rules of selecting the step-size have been discussed since the arising of stochastic approximation methods. Liu and Yang \cite{liuyang} introduced a new self-adaptive method for solving variational inequalities with Lipschitz continuous and quasimonotone mapping (or Lipschitz continuous mapping without monotonicity) in real Hilbert space. The method is defined as:
\[\left\{ {\begin{array}{*{20}{l}}
{{x_1} \in H,}\\
{{y_k} = {P_C}\left( {{x_k} - {\lambda _k}A{x_k}} \right),}\\
{{\lambda _{k + 1}} = \left\{ {\begin{array}{*{20}{l}}
{\min \left\{ {\frac{{\mu \left\| {{x_k} - {y_k}} \right\|}}{{\left\| {A{x_k} - A{y_k}} \right\|}},{\lambda _k} } \right\}}&{,{\rm{if   }}A{x_k} - A{y_k} \ne 0,}\\
{{\lambda _k} ,}&{otherwise,}
\end{array}} \right.}\\
{{x_{k + 1}} = {y_k} - {\lambda _k}\left( {A{y_k} - A{x_k}} \right).}
\end{array}} \right.\]
They introduced the adaptive step-size to modify the gradient method and get the weak convergence without knowing the Lipschitz constant. And another popular step-size is Armijo like step-size, which is a fundamental step size selection technique in optimization algorithms. It ensures sufficient decrease in the objective function while balancing computational efficiency.

Recently Mewomo {et al. } \cite{two 2024} used the two-step inertial acceleration strategy and the adaptive step-size to introduce a two-step inertial Tseng method involving quasimonotone and  uniformly continuous operator. Their algorithm is given by:

\begin{breakablealgorithm}
\renewcommand{\thealgorithm}{1}
%%%%%%%%%%%%%%%%%%%%%%%%%%%%%%%%%%%%%%%%%%%%%%%%%%%%%%%%%%%%%%%%%%%%%%%%%%%%%%%%%%%%%%%%%%%%%%%%%%
\caption{}\label{algo:myalg}
%%%%%%%%%%%%%%%%%%%%%%%%%%%%%%%%%%%%%%%%%%%%%%%%%%%%%%%%%%%%%%%%%%%%%%%%%%%%%%%%%%%%%%%%%%%%
\begin{algorithmic}%\label{algo:myalg}
\\ \hspace*{0.02in} {\bf Initialization:} \emph{Given} $\gamma>0, l \in(0,1), \mu \in(0,1)$. \emph{Let} $ x_{-1},x_{0},x_{1} \in \mathcal{H}$ 
\emph{be arbitrary points and given} $x_{k-2},x_{k-1},x_{k}$  \\    % 算法的输入， \hspace*{0.02in} 用来控制位置，同时利用 \\ 进行换行$x_{k-2},x_{k-1},x_{k}$
\hspace*{0.02in} {\bf Iterative: }Calculate $x_{k+1}$ as follows \\
\hspace*{0.02in} {\bf Step 1: }Set $w_{k}=x_{k}+\alpha(x_{k}-x_{k-1})+\beta(x_{k-1}-x_{k-2})$ and compute
$$
y_{k}=P_{C}\left(w_{k}-\lambda_{k} A w_{k}\right),
$$
\\
%\emph{where} $\lambda_{k}=min(\lambda_{k}^{(1)},\lambda_{k}^{(2)})$\\
%\begin{equation}\label{t1}
%\lambda_{k}^{(1)}=\begin{cases}\min \left\{\frac{\mu\left(\left\|w_{k}-y_{k}\right\|\right)}
%{A\left(w_{k}\right)-A\left(y_{k}\right)}, \lambda_{n-1}\right\}, & \text { if } A\left(w_{k}\right)-A\left(y_{k}\right)\neq0, \\ \lambda_{n-1}, & \text { otherwise }.\end{cases}
%\end{equation}
\emph{where} $\lambda_{k} = \gamma l^{m_{k}}$ \emph{and} $m_{k}$ \emph{is the smallest nonnegative integer m such that}
\begin{equation}\label{t2}
  \lambda_{k}\left\|A w_{k}-A y_{k}\right\| \leq \mu\left\|w_{k}-y_{k}\right\|.
\end{equation}
\emph{If} $y_{k} = w_{k}$ \emph{then stop:} $y_{k}$ \emph{is a solution of the problem (VIP). Otherwise,}\\
\hspace*{0.02in} {\bf Step 2:} Compute
$$
x_{k+1}=y_{k}-\lambda_{k}(Ay_{k}-w_{k}).
$$
\\
\emph{Set} $k:= k+1$ \emph{and go to \textbf{Step 1}.}
\end{algorithmic}
\end{breakablealgorithm}
They proved weak convergence of the sequence generated by their proposed algorithm. \\

From different angle,  Peng et al.  \cite{chou} used the Armijo-like condition to  proposed a modified Tseng method for solving pseudo-monotone variational inequality problems and the fixed point problems of the demi-contractive mappings. Their  algorithm is defined as:

\begin{breakablealgorithm}
\renewcommand{\thealgorithm}{2}
%%%%%%%%%%%%%%%%%%%%%%%%%%%%%%%%%%%%%%%%%%%%%%%%%%%%%%%%%%%%%%%%%%%%%%%%%%%%%%%%%%%%%%%%%%%%%%%%%%
\caption{}\label{2}
%%%%%%%%%%%%%%%%%%%%%%%%%%%%%%%%%%%%%%%%%%%%%%%%%%%%%%%%%%%%%%%%%%%%%%%%%%%%%%%%%%%%%%%%%%%%
\begin{algorithmic}%\label{algo:myalg}
\\ \hspace*{0.02in} {\bf Initialization:} \emph{Given} $k=1,\rho>0,\lambda_0>0,\gamma>0, l \in(0,1), \mu \in(0,1)$. \emph{Let} $x_{0},x_{1} \in \mathcal{H}$ \\
\emph{be arbitrary points}    \\    % 算法的输入， \hspace*{0.02in} 用来控制位置，同时利用 \\ 进行换行
%\hspace*{0.02in} {\bf Iterative: Calculate $x_{k+1}$ as follows} \\
\hspace*{0.02in} {\bf Step 1:} Given the current iterate $x_{k-1},x_{k}$

$$w_{k}=x_{k}+\alpha_k(x_{k}-x_{k-1}),$$
where
$${\alpha _k} = \left\{ {\begin{array}{*{20}{c}}
{\min \left\{ {\frac{{{\xi _k}}}{{\left\| {{x_k} - {x_{k - 1}}} \right\|}},\rho } \right\},}&{{\rm{if  }}{x_k} \ne {x_{k - 1}}},\\
\rho, &{{\rm{otherwise.}}}
\end{array}} \right.$$

\hspace*{0.02in} {\bf Step 2:} Compute
$$
y_{k}=P_{C}\left(w_{k}-\lambda_{k} A w_{k}\right),
$$
\\
\emph{where} $\lambda_{k}=min(\lambda_{k}^{(1)},\lambda_{k}^{(2)})$,\\
\begin{equation}\label{t1}
\lambda_{k}^{(1)}=\begin{cases}\min \left\{\frac{\mu\left(\left\|w_{k}-y_{k}\right\|\right)}
{A\left(w_{k}\right)-A\left(y_{k}\right)}, \lambda_{k-1}\right\}, & \text { if } A\left(w_{k}\right)-A\left(y_{k}\right)\neq0, \\ \lambda_{k-1}, & \text { otherwise }.\end{cases}
\end{equation}
\emph{where} $\lambda_{k}^{(2)} = \gamma l^{m_{k}}$ \emph{and} $m_{k}$ \emph{is the smallest nonnegative integer m such that}
\begin{eqnarray}\label{t2}
  \lambda_{k}\left\|A w_{k}-A y_{k}\right\| &\leq& \mu\left\|w_{k}-y_{k}\right\|,\\
  q_k & = & (1-\beta_k)z_k+\beta_kUz_k,
\end{eqnarray}
where $z_k = y_k +\lambda_k (Aw_k-Ay_k)$.\\
    If $w_k=y_k=q_k$, then stop $w_k$ is a solution of $w_k \in VI(C,A)\cap Fix(U)$. Otherwise\\
\hspace*{0.02in} {\bf Step 3:} Compute
$$
x_{k+1}=\eta_k D_k +(1-\eta_k)q_k,
$$
\\
\emph{where} $D_k = (1-\theta_k)f(x_k)+\theta_kf(x_{k-1}),$\emph{ Set }$k:=k+1$ \emph{and go to \textbf{Step 1}.}
\end{algorithmic}
\end{breakablealgorithm}

 In the step2 of Algorithm 2, Peng et al choose the step-size as the minimum of the adaptive step-size and the   Armijo  rule step size. In this paper,  inspired by the work of \cite{chou} and \cite{two 2024}, we have the following contributions: 

\begin{itemize}
\item We propose a novel algorithm that incorporates a two-step inertial technique to accelerate convergence. Unlike traditional one-step inertial methods, our approach leverages historical information from two preceding iterations, enhancing the algorithm's momentum and convergence speed.
\item Our work extends the applicability of Tseng’s extragradient method to quasimonotone VIPs, a broader and more general class of problems compared to the monotone or strongly monotone cases typically studied. This extension is significant because quasimonotonicity covers a wider range of practical problems while requiring weaker assumptions for convergence.
\end{itemize}

%In the second part, we provide several necessary lemmas to establish the theoretical foundation

% We propose a novel algorithm that incorporates a two-step inertial technique to accelerate convergence. Unlike traditional one-step inertial methods, our approach leverages historical information from two preceding iterations, enhancing the algorithm's momentum and convergence speed. This modification is inspired by recent findings that multi-step inertial methods can outperform their one-step counterparts in optimization tasks.Our work extends the applicability of Tseng’s extragradient method to quasimonotone VIPs, a broader and more general class of problems compared to the monotone or strongly monotone cases typically studied. This extension is significant because quasimonotonicity covers a wider range of practical problems while requiring weaker assumptions for convergence.

%Write the introduction of the paper here. Clearly mention the motivation behind the paper and statement of the problem.  
%
%
%Please order the references by the family names of the authors as ordered in the references. Like, the first cited article is
%\cite{hernandez2015survey}. Then, next is this \cite{eichfelder2012vector}. Later, the others are in the sequences like this \cite{alexandre2023validated, jiang2019generalized,hiriart2013convex}. 
%
%
%This paper is organized as follows.
%In Section \ref{section2}, preliminaries and a clear problem statement are provided. Write an outline or delineation of the paper here. 

\section{Preliminaries} \label{section2}

%\begin{definition} 
%A definition here. 
%\end{definition}
%
%Always use the label-and-refer style to refer to any equation number, section number, figure number, etc. Please do not manually type any equation number, section number, figure number, table number, etc. 
%
%
%\subsection{A subsection title is not in capitalized each word mode}

In this section, we review the definitions and lemmas required for this article.
\begin{definition}\label{def}
  Let $H$ be a real Hilbert space. An operator $A: H \rightarrow H$ is said to be:\\
(i) \emph{Lipschitz continuous} on $H$, if there exist a constant $L > 0$ such that
    \begin{equation}
      ||Ax-Ay||\leq L||x-y||, \forall x,y \in H.\notag
    \end{equation}\\
(ii) uniformly continuous, if for every $\epsilon > 0$, there exists $\delta = \delta(\epsilon)> 0$, such that 
    \begin{equation}
      ||Ax-Ay||< \epsilon ,\text{ whenever} ~~||x-y||< \delta,\text{ } \forall x,y \in H; \notag
    \end{equation}\\
(iii) sequentially weakly-strongly continuous, if for each sequence ${x_k}$, we have $x_k\rightharpoonup x \in H $  implies that $ Ax_k \rightarrow Ax \in H$; \notag\\
(iv) sequentially weakly continuous, if for each sequence ${x_n}$,we have $x_n\rightharpoonup x \in H $  implies that $ Ax_k \rightarrow Ax \in H$ ;\notag\\
(v) monotone, if $\langle Ax-Ay,x-y \rangle \geq 0, \forall x,y \in H ;$\notag\\
(vi) $\alpha$-strongly pseudomonotone, if there exists $\alpha > 0$ such that 
    \begin{equation}
      \langle Ax,y-x \rangle \geq 0 \Rightarrow \langle Ay,y-x \rangle \geq \alpha||x-y||^2,\text{ } \forall x,y \in H;\notag
    \end{equation}\\
(vii) pseudomonotone, if
    \begin{equation}
      \langle Ax,y-x \rangle \geq 0 \Rightarrow \langle Ay,y-x \rangle \geq 0,\text{ } \forall x,y \in H;\notag
    \end{equation}\\
(viii) quasimonotone, if
    \begin{equation}
      \langle Ax,y-x \rangle > 0 \Rightarrow \langle Ay,y-x \rangle \geq 0,\text{ } \forall x,y \in H;\notag
    \end{equation}\\
\end{definition}
From the definition above, the following implications hold: $( v)$ $\Longrightarrow$ $( vii)$ $\Longrightarrow$ $(viii)$ but the converse is not true in general. 

Let $S_D$ be the solution set of the dual formulation of the VIP (1.1) defined as: find $x^*\in C$ such that

$$\langle Az,\:z-x^{*}\rangle\geq0,\:\forall z\in C.$$

Then, $S_D$ is a closed and convex subset of $C$, and since $A$ is continuous and $C$ is
convex, we have that $S_D\subset S.$
We have the following result on the solution set of the dual VIP .

\begin{lemma}\label{p1}
  Let $x,y,z\in H$ and $a,b\in\mathbb{R}$. Then
  \begin{equation}
    \begin{split}
       ||(1+a)x-(a-b)y-bz||^2 =&(1+a)||x||^2-(a-b)||y||^2-b||z||^2\\\notag
                               & +(1+a)(a-b)||x-y||\\\notag
                               & +b(1+a)||x-z||^2-b(a-b)||y-z||^2 .\\\notag
    \end{split}
  \end{equation}
\end{lemma}

\begin{lemma}\label{p2}
  Let $C$ be a nonempty closed subset of a real Hilbert space $H$. $\forall v\in H$ and $\forall z\in C$, we have
  \begin{equation}
    z = P_{C}v \Leftrightarrow \langle v-z,z-y\rangle\geq0, \forall y \in C.\notag
  \end{equation}
\end{lemma}

\begin{lemma}\label{u-v}
  The following identities hold for all $u,v\in H$ :
    \begin{equation}
      2\langle u,v \rangle =\|u\|^2+\|v\|^2-\|u-v\|^2=\|u+v\|^2-\|u\|^2-\|v\|^2.\notag
    \end{equation}
\end{lemma}

\begin{lemma}
  (\cite{lema60}) Suppose either\\
  (a) $A$ is pseudomonotone on $C$ and $S\neq\emptyset$;\\
  (b) $A$ is the gradient of $G$, where $G$ is a differential quasiconvex function on an
 open set $K\supset C$ and attains its global minimum on $C$;\\
  (c) $A$ is quasimonotone on $C$, $A\neq 0 $ on $C$ and $C$ is bounded;\\
  (d) $A$ is quasimonotone on C, $A\neq 0 $ on C and there exists a positive number $r$ such that, for every $x\in C$ with $||x|| \geq r$, there exists $y \in C$ such that $||y|| \leq r$ and $\langle Ax,y-x\rangle \leq 0$;\\
  (e) $A$ is quasimonotone on $C$, $int C \neq \emptyset$ and there exists $x^* \in S$ such that $Ax^*\neq 0$,
\end{lemma}
then $S_D$ is nonempty.

\begin{lemma}
  (\cite{lema16}) Let $H$ be a Hilbert space and $A : H \rightarrow H$ be a uniformly continuous operator. Suppose $x\in H$ and $\psi \geq \sigma > 0$. The following inequality
 holds:
 $$
 \frac{\|x-P_C(x-\psi Ax)\|}{\psi}\leq\frac{\|x-P_C(x-\sigma Ax)\|}{\sigma} .
 $$

\end{lemma}

\section{PROPOSE THE ALGORITHM}

The subsequent section outlines the presentation of the algorithm. To ensure the weak convergence properties of the algorithm are valid, we propose the following assumptions.

\begin{assumption}
~\\
          (i)  $S_{D}\neq \emptyset$; \\
         (ii)  $A$ is Lipschitz continuous on $H$; \\
        (iii)  $A$ satisfies the following condition: 
               whenever ${x_{k}}\subset C$ and $x_{k}\rightharpoonup v^{*}$, 
              one has $\left\| {A{v^*}} \right\| \le \lim \mathop {\inf }\limits_{n \to \infty } \left\| {A{x_k}} \right\|$; \\
         (iv)  $A$ is quasimonotone on $H$;  \\
          (v)  The set $\{v\in C: Av = 0\} \backslash S_{D}$ is a finite set.
\end{assumption}

\begin{assumption}
  Assume that $\alpha$, $\beta$ and $\mu$ meet the following conditions:\\
  (a) $0\leq \alpha \leq \frac{1-\mu}{3+\mu}$;\\
  (b) $max\{2\alpha (\frac{1-\mu}{3+\mu})-(1-\alpha),\frac{1}{2}[\alpha(1+\mu)-(\frac{(1-\mu)(1-\alpha)^2}{1+\alpha})]\}<\beta\leq0$;\\
  (c) $2\alpha^2\mu-(1-3\alpha)+\mu(1-\alpha)-\beta(4\alpha+3-\mu)+2\mu\beta^2<0$.\\
\end{assumption}

%In the following part, we will propose an algorithm for solving VIP().
Under the assumptions above, we propose an algorithm for solving VIP (\ref{In1}):

\begin{breakablealgorithm}
\renewcommand{\thealgorithm}{3}
%%%%%%%%%%%%%%%%%%%%%%%%%%%%%%%%%%%%%%%%%%%%%%%%%%%%%%%%%%%%%%%%%%%%%%%%%%%%%%%%%%%%%%%%%%%%%%%%%%
\caption{}\label{algo:myalg}
%%%%%%%%%%%%%%%%%%%%%%%%%%%%%%%%%%%%%%%%%%%%%%%%%%%%%%%%%%%%%%%%%%%%%%%%%%%%%%%%%%%%%%%%%%%%
\begin{algorithmic}%\label{algo:myalg}
\\ \hspace*{0.02in} {\bf Initialization:} \emph{Given} $\gamma>0, l \in(0,1), \mu \in(0,1)$. \emph{Let} $ x_{-1},x_{0},x_{1} \in \mathcal{H}$ 
\emph{be arbitrary points and given} $x_{k-2},x_{k-1},x_{k}$    \\    % 算法的输入， \hspace*{0.02in} 用来控制位置，同时利用 \\ 进行换行
\hspace*{0.02in} {\bf Iterative: }Calculate $x_{k+1}$ as follows \\
\hspace*{0.02in} {\bf Step 1: }Set $w_{k}=x_{k}+\alpha(x_{k}-x_{k-1})+\beta(x_{k-1}-x_{k-2})$ and compute
$$
y_{k}=P_{C}\left(w_{k}-\lambda_{k} A w_{k}\right),
$$
\\
\emph{where} $\lambda_{k}=min(\lambda_{k}^{(1)},\lambda_{k}^{(2)})$\\
\begin{equation}\label{t1}
\lambda_{k}^{(1)}=\begin{cases}\min \left\{\frac{\mu\left(\left\|w_{k}-y_{k}\right\|\right)}
{||Aw_{k}-Ay_{k}||}, \lambda_{k-1}\right\}, & \text { if } Aw_{k}-Ay_{k}\neq0, \\ \lambda_{k-1}, & \text { otherwise }.\end{cases}
\end{equation}
$\lambda_{k}^{(2)} = \gamma l^{m_{k}}$ \emph{and} $m_{k}$ \emph{is the smallest nonnegative integer m such that}
\begin{equation}\label{t2}
  \lambda_{k}^{(2)}\left\|A w_{k}-A y_{k}\right\| \leq \mu\left\|w_{k}-y_{k}\right\|.
\end{equation}
\emph{If} $y_{k} = w_{k}$ \emph{then stop:} $y_{k}$ \emph{is a solution of the problem (VIP). Otherwise,}\\
\hspace*{0.02in} {\bf Step 2:} Compute
$$
x_{k+1}=y_{k}-\lambda_{k}(Ay_{k}-Aw_{k}),
$$
\\
\emph{Set} $k:= k+1$ \emph{and go to \textbf{Step 1}.}
\end{algorithmic}
\end{breakablealgorithm}

%\begin{figure}[H]
%\centering
%\includegraphics[scale=0.6]{sample_figure}
%\caption{A sample figure}
%\end{figure}  

%
%\begin{table}[h]
%\centering
%\caption{Experimental data}
%\label{table9.1}
%\resizebox{0.3\textwidth}{!}{
%\begin{tabular}{cccccc}
%\thickhline
%$x_{1}$ & $x_{2}$ & Execution Time (sec) \\
%\hline
%0.17  & 0.57  & 21.76 \\         
%0.30  & 0.30  & 0.00  \\          
%0.72  & 0.72  & 3.01  \\         
%0.24  & 0.04  & 0.56  \\         
%0.75  & 0.35  & 1.24  \\  
%\thickhline
%\end{tabular}} 
%\end{table}

\section{Convergence analysis} \label{sec:Discussion}
\begin{lemma}
  Under the (i)–(iv) of Assumption 1, the sequence ${\lambda_{k}}$ generated Algorithm \ref{algo:myalg} satisfying 
  \begin{equation}\label{minlambda}
    \min\{\frac{l\mu}{L},\lambda_{0}\}\leq\lambda_{k}\leq\lambda_{0}.
  \end{equation}
\end{lemma}
\begin{proof}
  Firstly, it is obvious that ${\lambda^{(1)}_{k}}$ is a monotonically decreasing sequence. Since $A$ is a Lipschitz continuous mapping with constant $L>0$,  in the case of $Aw_k-Ay_k\neq0$, we have
    \begin{equation}
      \frac{{\mu \left\| {{w_k} - {y_k}} \right\|}}{{||A{w_k} - A{y_k}||}} \ge \frac{\mu }{L}\frac{{\left\| {{w_k} - {y_k}} \right\|}}{{\left\| {{w_k} - {y_k}} \right\|}} = \frac{\mu }{L} \notag 
    \end{equation}
\end{proof}
which implies that $0<\min(\frac{l\mu}{L},\lambda_0)<\min(\frac{\mu}{L},\lambda_0)\leq \lambda^{(1)}_{k} \leq \lambda_0$.
On the other hand, from the definition of ${\lambda^{(2)}_{k}}$ we have
    \begin{equation}
        \frac{\lambda_{k}^{(2)}}{l}\begin{Vmatrix}Aw_{k}-Ay_{k}\end{Vmatrix}\geq
        \mu\begin{Vmatrix}w_{k}-y_{k}\end{Vmatrix} \notag
    \end{equation}
combining this with $A$ is Lipschitz continuous on $\mathcal{H}$, we obtain
    \begin{equation}
      \frac{\lambda_k^{(2)}}{l}L\left\|w_k-y_k\right\|\geq\mu\left\|w_k-y_k\right\| \notag 
    \end{equation}
so ${\lambda^{(2)}_{k}}\geq \frac{l\mu}{L}$, therefore $\min\{\frac{l\mu}{L},\lambda_0\}\leq \lambda_k \leq \lambda_0$.
\begin{lemma}\label{lm42}
Suppose Assumption 3.1 (i) and (ii) and Assumption 3.2 hold. Then, the sequence $\{x_{k}\}$ generated by Algorithm \ref{algo:myalg} is bounded.
%Under the assumptions of Assunmptions 1 (i), (ii) and Assunmption 2, it can be obtained that the sequence ${x_{k}}$ generated by Algorithm (\ref{algo:myalg}) is bounded.
\end{lemma}

\begin{proof}
  Let $q\in S_{D}$. First of all, we estimate  $||x_{k+1}-q||^2$. It can be obtained from the definition of $x_{k+1}$ that:
\begin{align}\label{l42p1}
||x_{k+1}-q||^2 =&||y_{k}-\lambda_{k}(Ay_{k}-Aw_{k})-q||^2
\notag\\
=&||y_{k}-q||^2 +\lambda_{k}^2||Ay_{k}-Aw_{k}||^2-2\lambda_{k}\langle Ay_{k}-Aw_{k},y_{k}-q\rangle
\notag\\
=&||w_{k}-q||^2+||w_{k}-y_{k}||^2+2\langle y_{k}-w_{k},w_{k}-q\rangle+\lambda_{k}^2||Ay_{k}-Aw_{k}||^2
\notag\\
&-2\lambda_{k}\langle Ay_{k}-Aw_{k},y_{k}-q\rangle
\notag\\
=&||w_{k}-q||^2+||w_{k}-y_{k}||^2-2\langle y_{k}-w_{k},y_{k}-w_{k}\rangle+
2\langle y_{k}-w_{k},y_{k}-q\rangle
\notag\\
&+\lambda_{k}^2||Ay_{k}-Aw_{k}||^2-2\lambda_{k}\langle Ay_{k}-Aw_{k},y_{k}-q\rangle
\notag\\
=&||w_{k}-q||^2-||w_{k}-y_{k}||^2+2\langle y_{k}-w_{k},y_{k}-q\rangle+\lambda_{k}^2||Ay_{k}-Aw_{k}||^2
\notag\\
&-2\lambda_{k}\langle Ay_{k}-Aw_{k},y_{k}-q\rangle.
\end{align}
From the definition of $y_{k}$, and the fact that $S_{D}\subset C$,
%see $y_{k}$ as $z$ and $P_{C}\left(x_{k}-\lambda_{k} A x_{k}\right)$ as $v$, since $q\in S_{D}\subset C$, 
 it can be obtained by Lemma \ref{p2} that:
$$
  \langle w_{k}-\lambda_{k}Aw_{k}-y_{k},y_{k}-q\rangle\geq0,
$$
that is
\begin{equation}\label{l42p2}
  \langle y_k-w_k,y_k-q\rangle\leq-\lambda_k\langle Aw_k,y_k-q\rangle.
\end{equation}
Appling (\ref{l42p2}) to (\ref{l42p1}), we will have
\begin{align}\label{l43p3}
||x_{k+1}-q||^2 \leq & ||w_{k}-q||^2-||w_{k}-y_{k}||^2-2\lambda_{k}\langle Aw_{k},y_{k}-q\rangle
\notag\\
&+\lambda_{k}^2||Ay_{k}-Aw_{k}||^2-2\lambda_{k}\langle Ay_{k}-Aw_{k},y_{k}-q\rangle
\notag\\
=&||w_{k}-q||^2-||w_{k}-y_{k}||^2+\lambda_{k}^2||Ay_{k}-Aw_{k}||^2
\notag\\
&-2\lambda_{k}\langle Ay_{k},y_{k}-q \rangle.
\end{align}
Since $q\in S_{D}, \lambda_{k}>0$ and $y_{k}\in C$, we have that 
\begin{equation}\label{l42p4}
  \lambda_{k}\langle Ay_{k},y_{k}-q \rangle \geq0.
\end{equation}
Therefore, from (\ref{l43p3}) we can have
\begin{equation}\label{l42p5}
 ||x_{k+1}-q||^2 \leq ||w_{k}-q||^2-||w_{k}-y_{k}||^2+\lambda_{k}^2||Ay_{k}-Aw_{k}||^2.
\end{equation}
From the definition of  $\lambda_{k}$, if ${Aw_{k}-Ay_{k}}\neq0$, then we have
$$
\lambda_{k}^{(1)} = \min\{\frac{\mu||w_{k}-y_{k}||}{||Aw_{k}-Ay_{k}||},\lambda_{k-1}\}
$$
which means:
$$
\lambda_{k}^{(1)} \leq\frac{\mu||w_{k}-y_{k}||}{||{Aw_{k}-Ay_{k}}||},
$$
$$
\lambda_{k}^{(1)}||{Aw_{k}-Ay_{k}}||\leq\mu||w_{k}-y_{k}||.
$$
Also, we have $\lambda_{k}^{(2)}\left\|A w_{k}-A y_{k}\right\| \leq \mu\left\|w_{k}-y_{k}\right\|$. Therefore,
\begin{equation}\label{l42p6}
  \lambda_{k}\left\|A w_{k}-A y_{k}\right\| \leq \mu\left\|w_{k}-y_{k}\right\|.
\end{equation}
Substituting  (\ref{l42p6}) into (\ref{l42p5}), we have
\begin{align}\label{l42p7}
  ||x_{k+1}-q||^2 \leq&||w_{k}-q||^2-||w_{k}-y_{k}||^2+\mu^2||w_{k}-y_{k}||^2
  \notag\\
  =& ||w_{k}-q||^2-(1-\mu^2)||w_{k}-y_{k}||^2.
\end{align}
Observe that
\begin{align}
||x_{k+1}-y_{k}||=&||y_{k}-\lambda_{k}(Ay_{k}-Aw_{k})-y_{k}|| \notag\\
  = & \lambda_{k}||Ay_{k}-Aw_{k}||\notag \\
  \leq & \mu||w_{k}-y_{k}||,\notag
\end{align}
and
\begin{align}\label{l42p8}
||x_{k+1}-w_{k}||\leq &||x_{k+1}-y_{k}||+||w_{k}-y_{k}||\notag \\
  \leq & \mu||w_{k}-y_{k}||+||w_{k}-y_{k}|| \notag\\
  = &  (1+\mu)||w_{k}-y_{k}||.
\end{align}
Hence, using (\ref{l42p8}) into (\ref{l42p7}), we obtain that
\begin{equation}\label{l42p9}
  -||w_{k}-y_{k}||^2\leq-\frac{1}{(1+\mu)^2}||x_{k+1}-w_{k}||^2.
\end{equation}
Applying (\ref{l42p9}) into (\ref{l42p7}), we have
\begin{equation}\label{l42p10}
  ||x_{k+1}-q||^2 \leq ||w_{k}-q||^2 - (\frac{1-\mu}{1+\mu})||x_{k+1}-w_{k}||^2.
\end{equation}
Also, 
\begin{align}
  w_{k}-q = & x_{k}+\alpha(x_{k}-x_{k-1})+\beta(x_{k-1}-x_{k-2})-q\notag\\
  = & (1+\alpha)(x_{k}-q)-(\alpha-\beta)(x_{k-1}-q)-\beta(x_{k-2}-q).\notag
\end{align}
%See $(1+\alpha)(x_{k}-q)$ as $(1+a)x$, $(\alpha-\beta)(x_{k-1}-q)$ as $(a-b)y$ and $\beta(x_{k-2}-q)$ as $bz$,\\
Using Lemma \ref{p1}, we get
\begin{align}\label{wn-q}
  ||w_{k}-q||^2 =& ||(1+\alpha)(x_{k}-q)-(\alpha-\beta)(x_{k-1}-q)-\beta(x_{k-2}-q)||^2 \notag \\
  = & (1+\alpha)||x_{k}-q||^2-(\alpha-\beta)||x_{k-1}-q||^2-\beta||x_{k-2}-q||^2 \notag\\
   & +(1+\alpha)(\alpha-\beta)||x_{k}-x_{k-1}||^2+\beta(1+\alpha)||x_{k}-x_{k-2}||\notag \\
   & -\beta(\alpha-\beta)||x_{k-1}-x_{k-2}||^2.
\end{align}
Furthermore, it can be obtained that
\begin{align}\label{xn+1-wn}
||x_{k+1}-w_{k}||^{2}=&\|x_{k+1}-[x_k+\alpha(x_k-x_{k-1})+\beta(x_{k-1}-x_{k-2})]\|^2  \notag \\
    =&\|x_{k+1}-x_{k}-\alpha(x_{k}-x_{k-1})-\beta(x_{k-1}-x_{k-2})\|^{2}\notag\\
    =&\|x_{k+1}-x_{k}\|^{2}-2\alpha\langle x_{k+1}-x_{k},x_{k}-x_{k-1}\rangle-2\beta\langle x_{k+1}-x_{k},x_{k-1}-x_{k-2}\rangle \notag\\
    &+\alpha^2||x_{k}-x_{k-1}||^2+2\alpha\beta\langle x_{k}-x_{k-1},x_{k-1}-x_{k-2}\rangle+\beta^2||x_{k-1}-x_{k-2}||\notag\\
    \geq&\|x_{k+1}-x_{k}\|^{2}-\alpha\|x_{k+1}-x_{k}\|^{2}-\alpha\|x_{k}-x_{k-1}\|^{2}+\alpha^{2}\|x_{k}-x_{k-1}\| \notag\\
    &-|\beta|\alpha\|x_{k}-x_{k-1}\|^2-|\beta|\alpha\|x_{k-1}-x_{k-2}\|^{2}+\beta^{2}\|x_{k-1}-x_{k-2}\|^{2} \notag\\
    &-|\beta|\|x_{k+1}-x_k\|^2-|\beta|\|x_{k-1}-x_{k-2}\|^2 \notag\\
    =&(1-|\beta|-\alpha)\|x_{k+1}-x_{k}\|^{2}+(\alpha^{2}-\alpha-|\beta|\alpha)\|x_{k}-x_{k-1}\|^{2} \notag\\
    &+(\beta^{2}-|\beta|-|\beta|\alpha)\|x_{k-1}-x_{k-2}\|^{2}.
\end{align}
Applying (\ref{wn-q}) and (\ref{xn+1-wn}) with (\ref{l42p10}) and note that $\beta <0$ we have
\begin{align}
||x_{k+1}-q||^{2}\leq&(1+\alpha)\|x_{k}-q\|^{2}-(\alpha-\beta)\|x_{k-1}-q\|^{2}-\beta\|x_{k-2}-q\|^{2}\notag\\
    &+(1+\alpha)(\alpha-\beta)\|x_k-x_{k-1}\|^2+\beta(1+\alpha)\|x_k-x_{k-2}\|^2\notag\\
    &-\beta(\alpha-\beta)\|x_{k-1}-x_{k-2}\|^2-\Big(\frac{1-\mu}{1+\mu}\Big)(1-|\beta|-\alpha)\|x_{k+1}-x_k\|^2\notag\\
    &-\Big(\frac{1-\mu}{1+\mu}\Big)(\alpha^2-\alpha-|\beta|\alpha)\|x_k-x_{k-1}\|^2\notag\\
    &-\Big(\frac{1-\mu}{1+\mu}\Big)(\beta^2-|\beta|-|\beta|\alpha)\|x_{k-1}-x_{k-2}\|^2\notag\\
    \leq &(1+\alpha)\|x_k-q\|^2-(\alpha-\beta)\|x_{k-1}-q\|^2-\beta\|x_{k-2}-q\|^2+\Big[(1+\alpha)(\alpha-\beta)\notag\\
    &-\Big(\frac{1-\mu}{1+\mu}\Big)(\alpha^2-\alpha+\beta\alpha)\Big]\|x_k-x_{k-1}\|^2\notag\\
    &-[\beta(\alpha-\beta)+\Big(\frac{1-\mu}{1+\mu}\Big)(\beta^2+\beta+\beta\alpha)]\|x_{k-1}-x_{k-2}\|^2\notag\\
    &-\Big(\frac{1-\mu}{1+\mu}\Big)(1+\beta-\alpha)\|x_{k+1}-x_k\|^2.
\end{align}
By rearranging the inequality above we obtain that:
\begin{align}\label{Rearange}
&\|x_{k+1}-q\|^{2}-\alpha\|x_{k}-q\|^{2}-\beta\|x_{k-1}-q\|^{2}+\Big(\frac{1-\mu}{1+\mu}\Big)(1+\beta-\alpha)\|x_{k+1}-x_{k}\|^{2}\notag\\
&\leq\|x_{k}-q\|^{2}-\alpha\|x_{k-1}-q\|^{2}-\beta\|x_{k-2}-q\|^{2}+\Big(\frac{1-\mu}{1+\mu}\Big)(1+\beta-\alpha)\|x_{k}-x_{k-1}\|^{2}\notag\\
&+[(1+\alpha)(\alpha-\beta)-\big(\frac{1-\mu}{1+\mu}\big)(\alpha^2-2\alpha+\beta\alpha+\beta+1)\big]\|x_k-x_{k-1}\|^2\notag\\
&-[\beta(\alpha-\beta)+\left(\frac{1-\mu}{1+\mu}\right)(\beta^2+\beta+\beta\alpha)]\|x_{k-1}-x_{k-2}\|^2.
\end{align}
Now, let  
$$
\Gamma_{k} = \|x_k-q\|^2-\alpha\|x_{k-1}-q\|^2-\beta\|x_{k-2}-q\|^2+\left(\frac{1-\mu}{1+\mu}\right)(1+\beta-\alpha)\|x_k-x_{k-1}\|^2.
$$
Then we can rewrite (\ref{Rearange}) as 
\begin{align}\label{l42p12}
  \Gamma_{k+1} \leq & \Gamma_k+[(1+\alpha)(\alpha-\beta)
   -\left(\frac{1-\mu}{1+\mu}\right)(\alpha^2-2\alpha+\beta\alpha+\beta+1)]\|x_k-x_{k-1}\|^2 \notag\\
  & -[\beta(\alpha-\beta)+\left(\frac{1-\mu}{1+\mu}\right)(\beta^2+\beta+\beta\alpha)]\|x_{k-1}-x_{k-2}\|^2. 
\end{align}
{\bf Claim:} $\Gamma_{k}\geq0$, $\forall k>1$. Since
\begin{align}\label{Tn=}
\Gamma_{k}
=&\|x_{k}-q\|^{2}-\alpha\|x_{k-1}-q\|^{2}-\beta\|x_{k-2}-q\|^{2}+\left(\frac{1-\mu}{1+\mu}\right)(1+\beta-\alpha)\|x_{k}-x_{k-1}\|^{2} \notag\\
\geq&\|x_{k}-q\|^{2}-2\alpha\|x_{k}-x_{k-1}\|^{2}-2\alpha\|x_{k}-q\|^{2}-\beta\|x_{k-2}-q\|^{2} \notag\\
&+\left(\frac{1-\mu}{1+\mu}\right)(1+\beta-\alpha)\|x_k-x_{k-1}\|^2 \notag\\
=&(1-2\alpha)\|x_{k}-q\|^{2}+\left[\left(\frac{1-\mu}{1+\mu}\right)(1+\beta-\alpha)-2\alpha\right]\|x_{k}-x_{k-1}\|^{2} \notag\\
&-\beta||x_{k-2}-q||^2.
\end{align}
From Assumption 2, we have that $\alpha<\frac{1}{2},\beta\leq0,\frac{\alpha(3+\mu)}{1-\mu}-1\leq\beta$ and $0<\alpha<\frac{1-\mu}{3+\mu}$, it can be conclude that $\Gamma_{k}\geq 0$, $\forall k>1$. Let
%\begin{align}
%\Gamma_{k+1}-\Gamma_{k}
%\leq&[(1+\alpha)(\alpha-\beta)-\left(\frac{1-\mu}{1+\mu}\right)(\alpha^{2}-2\alpha+\beta\alpha+\beta+1)]\|x_{k}-x_{k-1}\|^{2} \notag\\
%&-[\beta(\alpha-\beta)+\left(\frac{1-\mu}{1+\mu}\right)(\beta^2+\beta+\beta\alpha)]\|x_{k-1}-x_{k-2}\|^2 \notag\\
%=&-[(1+\alpha)(\alpha-\beta)-\left(\frac{1-\mu}{1+\mu}\right)(\alpha^2-2\alpha+\beta\alpha+\beta+1)](\|x_{k-1}-x_{k-2}\|^2-\|x_k-x_{k-1}\|^2) \notag\\
%&+[(1+\alpha)(\alpha-\beta)-\left(\frac{1-\mu}{1+\mu}\right)(\alpha^2-2\alpha+\beta\alpha+\beta+1)-\beta(\alpha-\beta) \notag\\
%&-\left(\frac{1-\mu}{1+\mu}\right)(\beta^2+\beta+\beta\alpha)]\|x_{k-1}-x_{k-2}\|^2.
%\end{align}

\begin{align}
  k_{1} :=& -[(1+\alpha)(\alpha-\beta)-\left(\frac{1-\mu}{1+\mu}\right)(\alpha^2-2\alpha+\beta\alpha+\beta+1)] \notag\\
  k_{2} :=& -[(1+\alpha)(\alpha-\beta)-\left(\frac{1-\mu}{1+\mu}\right)(\alpha^2-2\alpha+\beta\alpha+\beta+1)-\beta(\alpha-\beta) -\left(\frac{1-\mu}{1+\mu}\right)(\beta^2+\beta+\beta\alpha)]. \notag\\
\end{align}
Thus, we deduce from (\ref{l42p12}) that
\begin{equation}\label{gamma-gamma}
\Gamma_{k+1}-\Gamma_n\leq k_1(\|x_{k-1}-x_{k-2}\|^2-\|x_k-x_{k-1}\|^2)-k_2\|x_{k-1}-x_{k-2}\|^2.
\end{equation}
From Assumption 2(b), 
\begin{equation}
  \max\Big\{2\alpha (\frac{1-\mu}{3+\mu})-(1-\alpha),\frac{1}{2}[\alpha(1+\mu)-(\frac{(1-\mu)(1-\alpha)^2}{1+\alpha})]\Big\}<\beta\leq0, \notag 
\end{equation}
which implies that 
\begin{equation}\label{2bchange}
  \frac{1}{2}\left[\alpha(1+\mu)-\frac{(1-\mu)(1-\alpha)^2}{1+\alpha}\right]<\beta,
\end{equation}
which can conclude that $k_{1}>0$. \\
By Assumption 2(c) we can have $k_{2}>0$, thus, (\ref{gamma-gamma}) can be rewritten as
\begin{equation}\label{gammak1}
\Gamma_{k+1}+k_1\|x_k-x_{k-1}\|^2\leq\Gamma_k+k_1\|x_{k-1}-x_{k-2}\|^2-k_2\|x_{k-1}-x_{k-2}\|^2.
\end{equation}
Letting $\Gamma_k^{\prime}=\Gamma_k+k_1\|x_{k-1}-x_{k-2}\|^2.$ Then, $\Gamma_k^{^{\prime}}\geq0$,  $\forall k\geq1$. Therefore, we deduce from (\ref{gammak1}) that 
\begin{equation}\label{gamma'}
  \Gamma_{k+1}^{'}\leq\Gamma_{k}^{'},
\end{equation}
which implies that  the sequence ${\Gamma_{k+1}^{'}}$ is decreasing and bounded from below
and thus $\lim\limits_{n\rightarrow\infty}\Gamma_{k}^{'}$ exists. \\
Hence, by rearranging (\ref{gammak1}) and letting $k$ approach infinity, we have
%let $k$ on both sides of the (\ref{gammak1}) approach infinity, we can get 
\begin{equation}\label{lim0}
\lim_{k\rightarrow\infty}k_{2}||x_{k-1}-x_{k-2}||^2=0 \implies\lim_{k\rightarrow\infty}||x_{k-1}-x_{k-2}||=0.
\end{equation}
It is easy to see that 
\begin{align}\label{xn1wn}
\|x_{k+1}-w_{k}\|
=&\|x_{k+1}-x_{k}-\alpha(x_{k}-x_{k-1})-\beta(x_{k-1}-x_{k-2})\| \notag\\
\leq&\|x_{k+1}-x_{k}\|+\alpha\|x_{k}-x_{k-1}\|+\beta\|x_{k-1}-x_{k-2}\|.
\end{align} 
Consequently, by (\ref{lim0}) we have 
 %add (\ref{lim0}) into (\ref{xn1wn}), that is
\begin{align}\label{xn1wn2}
\lim_{k\rightarrow\infty}\|x_{k+1}-w_{k}\|=0.
%\leq&\|x_{k+1}-x_{k}\|+\alpha\|x_{k}-x_{k-1}\|+\beta\|x_{k-1}-x_{k-2}\|\rightarrow0.(k\rightarrow0)
\end{align} 
Furthermore,
\begin{align}
\|x_{k}-w_{k}\|
=&\|x_{k}-x_{k}-\alpha(x_{k}-x_{k-1})-\beta(x_{k-1}-x_{k-2})\|\notag\\
\leq&\alpha\|x_{k}-x_{k-1}\|+\beta\|x_{k-1}-x_{k-2}\|\rightarrow0, k\rightarrow\infty.
\end{align}
Observe that
\begin{align}\label{xn1wn->0}
\|x_{k+1}-w_{k}\|=&\quad\|y_{k}-\lambda_{k}(Ay_{k}-Aw_{k})-w_{k}\|\notag\\
\geq&\quad\|y_{k}-w_{k}\|-\lambda_{k}\|Ay_{k}-Aw_{k}\|\notag\\
\geq&\quad\|y_{k}-w_{k}\|-\mu\|y_{k}-w_{k}\|\notag\\
=&\quad(1-\mu)\|y_{k}-w_{k}\|.
\end{align}
Using (\ref{xn1wn2} )we deduce from (\ref{xn1wn->0}) that 
%Following (\ref{xn1wn2}) and (\ref{xn1wn->0}), we can get that
\begin{equation}\label{wnyn->0}
\lim_{k\rightarrow\infty}\|w_k-y_k\| = 0.
\end{equation}
Also,
\begin{equation}\label{xnyn}
\|x_k-y_k\|\leq\|x_k-w_k\|+\|w_k-y_k\|\to0, k\to\infty.
\end{equation}
Due to the existence of the limit $\Gamma_{k}^{'}$ and (\ref{lim0}), 
we can obtain that the limit of $\Gamma_{k}$  also exists and therefore, the sequence $\{\Gamma_{k}\}$ is bounded.\newline
Since $\lim\limits_{k\rightarrow\infty}||x_{k+1}-x_{k}||=0$, from the definition of $\Gamma_{k}$ we can get that 
\begin{equation}\label{lim2}
  \lim_{k\to\infty}(\|x_k-q\|^2-\alpha\|x_{k-1}-q\|^2-\beta\|x_{k-2}-q\|^2)
\end{equation}
exists. Due to the boundedness of $\{\Gamma_{k}\}$, we have from (\ref{Tn=}) that $\{x_{k}\}$ is bounded and $\{y_{k}\}$ and $\{w_{k}\}$ are bounded.
\end{proof}
%{\color{red} the following part need to be revised}
\begin{lemma}\label{case}
Let $\{x_{k}\}$ be the sequence generated by algorithm \ref{algo:myalg}, satisfying Assumptions 1(i)-(iv) and Assumptions 2(a)-(c), and assuming that $\lim \limits_{k\rightarrow\infty}w_{k}-y_{k} = 0$. 
If $v^*$is one of the weakly clustered points of ${y_k}$, then we have at least one of the following: $v^*\in S_D$ or $Av^*= 0$.
\end{lemma}
\begin{proof}
  From the lemma (\ref{lm42}) above, it can be concluded that $\{y_{k}\}$ is bounded. Therefore let $v^*$ be a weak cluster point of $\{y_{k}\}$.
  Hence, we denote $\{y_{k_{j}}\}$ as a subsequence of $\{y_{k}\}$ such that $y_{k_{j}}\rightharpoonup v^*\in C$.
  
  Now we discuss in two cases.
    \begin{case}
      First, we assume that $\lim \limits_{k\rightarrow\infty}||Ay_{k_{j}}||=0$. Consequently, 
      $\lim \limits_{k\rightarrow\infty}||Ay_{k_{j}}||=\liminf \limits_{k\rightarrow\infty}||Ay_{k_{j}}||=0$
      
      Through the assumption above that $y_{k_{j}}\rightharpoonup v^*\in C$ and $A$ satisfies Assumption 1(iii), that is
        \begin{equation}\label{0<avlim}
          0\leq||Av^*||\leq\liminf \limits_{k\rightarrow\infty}||Ay_{k_{j}}||=0.
        \end{equation}
      Which just implies that $Av^*=0$.
    \end{case}
    Now we consider the another situation.

    \begin{case}
        If\quad $\limsup\limits_{k\to\infty}\|Ay_{k_j}\|>0$. Without loss of generality, we take $\lim\limits_{k\to\infty}\|Ay_{k_j}\|=M_{1}>0$. It then follows that there exists a $\mathbf{K}\in\mathbb{K}$ such that $\|Ay_{k_j}\|>\frac{M_{1}}{2}$ for all $k\geq \mathbf{K}$ Since $y_{k_{j}}=P_C(w_{k_{j}}-\lambda_{k_{j}}Aw_{k_{j}})$, we have 
        
        $$
        \langle w_{k_{j}}-\lambda_{k_{j}}Aw_{k_{j}}-y_{k_{j}},x-y_{k_{j}}\rangle \leq 0,
        $$
        \begin{equation}\label{ca2}
          \frac{1}{\lambda_{k_{j}}}\langle w_{k_{j}}-y_{k_{j}},x-y_{k_{j}} \rangle
          +\langle Aw_{k_{j}},y_{k_{j}}-w_{k_{j}}\rangle\leq\langle Aw_{k_{j}},x-w_{k_{j}} \rangle.
        \end{equation}
      For the weak convergence of $w_{k_{j}}$, $w_{k_{j}}$ is bounded. Then, for the $A$ Lipschitz continuous, $Aw_{k_{j}}$ is bounded. By (\ref{wnyn->0}) we can get that 
      $$
      \|w_{k_{j}}-y_{k_{j}}\|\to0,k\to\infty;
      $$
      therefore $\|y_{k_{j}}\|$ is also bounded and through the Lemma(\ref{minlambda}) we get that 
      $\lambda_{k_{j}}\geq min\{\frac{l\mu}{L},\lambda_{0}\}$. Passing (\ref{ca2}) to the limit as 
      $k\rightarrow\infty$, that is
      \begin{equation}\label{liminf}
        \liminf \limits_{k\rightarrow\infty}\langle Aw_{k_{j}},x-w_{k_{j}} \rangle \geq 0,\forall x \in C.
      \end{equation}
      Observe that
      \begin{align}\label{aynkxynk}
        \langle Ay_{k_{j}},x-y_{k_{j}}\rangle  
        =& \langle Ay_{k_{j}}-Aw_{k_{j}},x-w_{k_{j}}\rangle+\langle Aw_{k_{j}},x-w_{k_{j}}\rangle\notag\\
        &+\langle Ay_{k_{j}},w_{k_{j}}-y_{k_{j}}\rangle.
      \end{align}  
      Thus we can get  $\lim \limits_{k\rightarrow\infty}||Aw_{k_{j}}-Ay_{k_{j}}||=0$, for the 
      $\lim \limits_{k\rightarrow\infty}||w_{k_{j}}-y_{k_{j}}||=0$ and the $L$ -Lipschitz continuity on $H$ of $A$.
      Together with (\ref{liminf}) and (\ref{aynkxynk}) which implies that 
      \begin{equation}\label{liminfaynk}
        \lim \limits_{k\rightarrow\infty}\langle Ay_{k_{j}},x-y_{k_{j}}\rangle \geq 0.
      \end{equation}
      If we suppose that $\limsup\limits_{k\rightarrow\infty}\langle Ay_{k_{j}},x-y_{k_{j}}\rangle>0$, then there exists a subsequence denoted by $y_{n_{k_{j}}}$ 
      such that $\lim\limits_{j\rightarrow\infty}\langle Ay_{n_{k_{j}}},x-y_{n_{k_{j}}} \rangle > 0$,
      that is there exists $j_{0}\in \mathbb{K}$ such that 
      \begin{equation}\label{liminfaynk}
        \langle Ay_{n_{k_{{j}}}},x-y_{k_{j}}\rangle \geq 0, \forall j > j_{0}.
      \end{equation}
      Through the quasimonotonicity(\ref{def}) of $A$, we get $\langle Ax,x-y_{n_{k_{j}}}\rangle\geq0$, as $j\rightarrow\infty$, we conclude that $v^*\in S_{D}$
      On the other hand we suppose that if $\limsup\limits_{k\rightarrow\infty}\langle Ay_{k_{j}},x-y_{k_{j}}\rangle=0$, plus (\ref{liminf}) implies that 
      \begin{equation}\label{liminfaynkcase2}
        \lim \limits_{k\rightarrow\infty}\langle Ay_{k_{j}},x-y_{k_{j}}\rangle = 0.
      \end{equation}
      Let $\epsilon_{k}:= |Ay_{k_{j}},x-y_{k_{j}}|+\frac{1}{k+1}$. Thus we obtain that 
      \begin{equation}\label{axye}
        \langle Ay_{k_{j}},x-y_{k_{j}}\rangle + \epsilon_{k} > 0, \forall k\geq 1.
      \end{equation}
      Furthermore, for each $k\geq \mathbf{K}$ we can get $Ay_{k_{j}}\neq0$. Defining that 
      \begin{equation}\label{q}
        r_{k_{j}} = \frac{Ay_{k_{j}}}{||Ay_{k_{j}}||^2}, \forall~k\geq \mathbf{K}, \notag
      \end{equation}
      then $\langle Ay_{k_{j}},r_{k_{j}}\rangle=1$ for each $k\geq \mathbf{K}$.
      Thus we can conclude from (\ref{axye}) that,
      \begin{equation}\label{ayxeqy}
        \langle Ay_{k_{j}},x+\epsilon_{k}r_{k_{j}}-y_{k_{j}}\rangle > 0,~ k\geq \mathbf{K}.\notag
      \end{equation}
      Since $A$  is quasimonotone on $H$, we get
      \begin{equation}\label{ayxeqy}
        \langle A(x+\epsilon_{k}r_{k_{j}}),x+\epsilon_{k}r_{k_{j}}-y_{k_{j}}\rangle \geq 0.
      \end{equation}
      Thus
      \begin{align}
      \langle Ax,x+\epsilon_{k}r_{k_{j}}-y_{k_{j}}\rangle =&\langle Ax-A(x+\epsilon_{k}r_{k_{j}}),x+\epsilon_{k}r_{k_{j}}-y_{k_{j}}\rangle\notag\\
      &+\langle A(x+\epsilon_{k}r_{k_{j}}),x+\epsilon_{k}r_{k_{j}}-y_{k_{j}}\rangle \rangle \notag \\
      \geq&\langle Ax-A(x+\epsilon_{k}r_{k_{j}}),x+\epsilon_{k}r_{k_{j}}-y_{k_{j}}\rangle \notag\\
      \geq&-\|Ax-A(x+\epsilon_kq_{n_k})\|\|x+\epsilon_kq_{n_k}-y_{k_j}\| \notag\\
      \geq&-\epsilon_{k}L\|r_{k_{j}}\|\|x+\epsilon_{k}r_{k_{j}}-y_{k_{j}}\| \notag\\
      =&-\epsilon_{k}L\frac{1}{\|Ay_{k_{j}}\|}\|x+\epsilon_{k}r_{k_{j}}-y_{k_{j}}\| \notag\\
      \geq&-\epsilon_{k}L\frac{2}{M_{1}}\|x+\epsilon_{k}r_{k_{j}}-y_{k_{j}}\|.
      \end{align}
      Observe that, tending $k\rightarrow\infty$, for $\{x+\epsilon_{k}r_{k_{j}}-y_{k_{j}}\}$ is bounded 
      and $\lim\limits_{k\rightarrow\infty}\epsilon_{k}=0$,
      we can conclude that $\langle Ax,x-v^*\rangle\geq0,~\forall x\in C$. This implies that $v^*\in S_D$.
    \end{case} 
\end{proof}
\begin{theorem}
  Suppose $\{x_{k}\}$ is a sequence generated by Algorithm \ref{algo:myalg}. Then
under Assumptions 1 and Assumptions 2 and $Ax\neq 0$, $\forall x \in C$. Then ${x_{k}}$ converges weakly to an element of $S_{D} \subset S$.
\end{theorem}
\begin{proof}
  Suppose $w_{w}(x_{k})$ is a set of weak cluster points of $\{x_{k}\}$. Which is 
    $$
    w_{w}(x_{k})\subset S_{D}.
    $$
  We take $v^*\in w_{w}(x_{k})$.Therefore, there exists a subsequence $\{x_{n_{k}}\}\subset\{x_{k}\}$ such that $\{x_{n_{k}}\}\rightarrow v^*$, 
  $k \rightarrow\infty$. Since $C$ is weakly closed, we have that $v^*\in C$.
  Furthermore, we can conclude that $Av^*\neq0$ for $Ax\neq0, ~\forall x\in C $.
  By Lemma (\ref{case}), we have $v^*\in S_{D}$. Hence $w_{w}(x_{k})\subset S_{D}$.
  By Lemma (\ref{lm42}), $\lim\limits_{n\rightarrow\infty}\Gamma_{k}$ exists and 
  $\lim\limits_{n\rightarrow\infty}||x_{k+1}-x_{k}||=0$, we have 
  \begin{equation}\label{lim}
   \lim\limits_{n\rightarrow\infty}[||x_{k}-q||^2-\alpha||x_{k-1}-q||^2-\beta||x_{k-1}-q||^2]
  \end{equation}
  exists for $\forall q \in S_{D}$.
  
  Next we show that $x_{k}\rightharpoonup x^*\in S_{D}$. Let $\{x_{n_{j}}\}$ and $\{x_{n_{m}}\}$ both $\subset \{x_{k}\}$ such that $x_{n_{j}}\rightharpoonup v^*, j\rightarrow\infty$ and $x_{n_{m}}\rightharpoonup x^*, m\rightarrow\infty$.
  
  Then we show that $x^*=v^*$, using (\ref{u-v}) observe that 
  \begin{equation}\label{2<>1}
    2\langle x_{k}, x^*-v^*\rangle = ||x_{k}-v^*||^2-||x_{k}-x^*||^2-||v^*||^2+||x^*||^2;
  \end{equation}
  
  \begin{equation}\label{2<>2}
    2\langle x_{k-1}, x^*-v^*\rangle = ||x_{k-1}-v^*||^2-||x_{k-1}-x^*||^2-||v^*||^2+||x^*||^2;
  \end{equation}
  and
  \begin{equation}\label{2<>3}
    2\langle x_{k-2}, x^*-v^*\rangle = ||x_{k-2}-v^*||^2-||x_{k-2}-x^*||^2-||v^*||^2+||x^*||^2;
  \end{equation}
  Therefore, 
  \begin{equation}\label{2<>4}
    2\langle -\alpha x_{k-1}, x^*-v^*\rangle = -\alpha||x_{k-1}-v^*||^2+\alpha||x_{k-1}-x^*||^2+\alpha||v^*||^2-\alpha||x^*||^2;
  \end{equation}
  and
  \begin{equation}\label{2<>5}
    2\langle -\beta x_{k-1}, x^*-v^*\rangle = -\beta||x_{k-1}-v^*||^2+\beta||x_{k-1}-x^*||^2+\beta||v^*||^2-\beta||x^*||^2;
  \end{equation}
  Addition of (\ref{2<>1}), (\ref{2<>4}) and (\ref{2<>5}) gives
  \begin{align}\label{addi}
     2\langle x_k-\alpha x_{k-1}-\beta x_{k-2},x^*-v^*\rangle \notag
     =& \left(\|x_k-v^*\|^2-\alpha\|x_{k-1}-v^*\|^2-\beta\|x_{k-2}-v^*\|^2\right) \notag \\
     &-(\|x_k-x^*\|^2-\alpha\|x_{k-1}-x^*\|^2-\beta\|x_{k-2}-x^*\|^2)  \notag\\
     &+(1-\alpha-\beta)(\|x^*\|-\|v^*\|^2). 
  \end{align}
  Through (\ref{lim2}) we know that
  \begin{equation}\label{lim23}
  \lim\limits_{n\to\infty}(\|x_k-x^*\|^2-\alpha\|x_{k-1}-x^*\|^2-\beta\|x_{k-2}-x^*\|^2)
  \end{equation}
  exists, and 
  \begin{equation}\label{lim24}
  \lim\limits_{n\to\infty}(\|x_k-v^*\|^2-\alpha\|x_{k-1}-x^*\|^2-\beta\|x_{k-2}-x^*\|^2)
  \end{equation}
  also exists, which implies with (\ref{addi}) that 
  $\lim\limits_{n\to\infty}\langle x_{k}-\alpha x_{k-1}-\beta x_{k-2},x^*-v^* \rangle$ exists.
  
  Consequently,
  \begin{align}\label{last}
  \langle v^*-\alpha v^*-\beta v^*, x^*-v^* \rangle
    =&\lim\limits_{j\rightarrow\infty} \langle x_{k_{j}}-\alpha x_{k_{j-1}}-\beta x_{n_{j-2}},x^*-v^* \rangle\notag\\
    =&\lim\limits_{j\rightarrow\infty} \langle x_{k}-\alpha x_{k-1}-\beta x_{k-2},x^*-v^* \rangle \notag\\
    =&\lim\limits_{j\rightarrow\infty} \langle x_{k_{m}}-\alpha x_{k_{m-1}}-\beta x_{k_{m-2}},x^*-v^* \rangle \notag\\
    =&\langle x^*-\alpha x^*-\beta x^*,x^*-v^* \rangle . 
  \end{align}
  Thus,
  \begin{equation}\label{llast}
    (1-\alpha-\beta)||x^*-v^*||^2 = 0.
  \end{equation}
  For $\beta\leq0<1-\alpha$, $1-\alpha-\beta \neq0$, we can conclude that $x^*=v^*$.
  
  Hence, we deduce that $\{x_{k}\}$ converges weakly to a point in $S_{D}$. This completes the proof.
\end{proof}

%\section*{Statements and Declarations} 
%Please declare any financial or personal or any other conflict of interest here. Else, please write here the following: 
%The authors declare that they have no conflict of interest, and 
%the manuscript has no associated data. 

%\section*{Acknowledgments.}
%Please acknowledge any funding or other supports here. 
%

\section{Conclusions}
This paper proposed a novel two-step inertial Tseng extragradient method for solving quasimonotone variational inequalities in real Hilbert spaces. The algorithm incorporates a dual step-size strategy, adaptively selecting between a self-adaptive rule and an Armijo-like rule at each iteration.The proposed method effectively addresses limitations of existing extragradient methods, particularly the need for Lipschitz constants and the computational burden of multiple projections.

%\bibliographystyle{abbrv}
%\bibliography{tos_bib}

\end{document}